\documentclass[11pt]{amsart}
\usepackage{graphicx}
\usepackage{amsmath, amssymb}
\usepackage{verbatim}
\usepackage{color}
\newtheorem{thm}{Theorem}[section]
\newtheorem{cor}[thm]{Corollary}
\newtheorem{conj.}[thm]{Conjecture}
\newtheorem{lem}[thm]{Lemma}
\newtheorem{prop}[thm]{Proposition}
\theoremstyle{definition}
\newtheorem{defn}[thm]{Definition}
\theoremstyle{remark}

\numberwithin{equation}{section}
\newtheorem{exa}[thm]{Example}

\newcommand{\Real}{\mathbb R}
\newcommand{\Comp}{\mathbb C}
\newcommand{\NN}{\mathbb N}

\newcommand{\RR}{\mathbb R}

\newcommand{\h}{\mathcal{H}}

\newcommand{\K}{\mathcal{K}}
\newcommand{\R}{\mathcal{R}}
\begin{document}

\title[Multipliers for Continuous Frames in Hilbert spaces]
{Multipliers for continuous frames in Hilbert spaces}%
\author[P. Balazs, D. Bayer and A. Rahimi]{P. Balazs$^\dagger$, D. Bayer$^\dagger$ and A. Rahimi$^*$,}
\address{$^\dagger$ Acoustics Research Institute, Austrian Academy of Sciences, Wohllebengasse 12-14, 1040 Wien, Austria.}
\email{peter.balazs@oeaw.ac.at}
\email{bayerd@kfs.oeaw.ac.at}
\address{$^*$ Department of Mathematics, University of Maragheh, P. O. Box 55181-83111, Maragheh, Iran.}
\email{asgharrahimi@yahoo.com}


\begin{abstract}
In this paper we examine the general theory of continuous frame multipliers in Hilbert space. These operators are a generalization of the widely used notion of (discrete) frame multipliers. Well-known examples include Anti-Wick operators, STFT multipliers or Calder\'on-Toeplitz operators. Due to the possible peculiarities of the underlying measure spaces, continuous frames do not behave quite as well as their discrete counterparts. Nonetheless, many results similar to the discrete case are proven for continuous frame multipliers as well, for instance compactness and Schatten class properties. Furthermore, the concepts of controlled and weighted frames are transferred to the continuous setting. 
\end{abstract}

\maketitle

\section{Introduction}

A discrete frame is a countable family of elements in a separable
Hilbert space which allows stable but not necessarily unique
decomposition of arbitrary elements into expansion of the frame
elements. The concept of generalization of frames was proposed by G.
Kaiser \cite{Gk} and independently by Ali, Antoine and Gazeau
\cite{Ali2} to a family indexed by some locally compact space endowed
with a Radon measure. These frames are known as continuous frames.
Gabardo and Han in \cite{Gb} called these frames \textit{Frames
associated with measurable spaces}, Askari-Hemmat, Dehghan and
Radjabalipour in \cite{Ra} called these frames \textit{generalized
frames}  and they are linked to
\textit{coherent states}  in mathematical physics \cite{Ali2}. For more studies, the
interested reader can also refer to \cite{Ali1, jpaxxl09, jpaxxl11, chui, For}.
\par
Bessel and frame multipliers were introduced by one of the authors
\cite{xxlmult1,peter2,peter3} for Hilbert spaces. For Bessel
sequences, the investigation of the operator $\mathbf{M}=\sum m_k
\langle f,\psi_{k}\rangle\varphi_{k}$, where the analysis
coefficients $\langle f,\psi_{k}\rangle$ are multiplied by a fixed
symbol $(m_k)$ before resynthesis (with $\varphi_{k}$), is very
natural. There are numerous applications of this
kind of operators. As a particular way to implement time-variant filters Gabor frame multipliers \cite{feic} are used, also known as Gabor filters \cite{matz}. Such operators find application in psychoacoustics \cite{xxllabmask1}, denoising \cite{majxxl10}, computational auditory scene analysis \cite{wanbro06}, virtual acoustics \cite{majxxl1} and seismic data analysis \cite{gigrheillama05}. On a more theoretical level
 Bessel multipliers of $p$-Bessel sequences in
Banach spaces are introduced in \cite{peterasghar}.
\par
Wavelet and Gabor frames are used very often in signal processing algorithm. Both systems are derived from a continuous frame transform.
For these two special systems continuous frame multipliers have been investigated as STFT-multipliers \cite{feic} or Anti-Wick operators \cite{coro05} respectively Calder\`on - Toeplitz operators \cite{now93, roch90}.
In this paper we investigate multipliers for continuous frames in the general setting, with some comments on the mentioned special cases in Section \ref{sec:STFTwavmult0}.

The paper is organized as follows. In Section 2, we collect a number of notions and preliminaries on continuous Bessel mappings and frames and their most basic properties and present some well-known examples. 
In Section 3, we define continuous Bessel and frame multipliers as generalizations of discrete Bessel and frame multipliers, develop their theory and prove a number of statements on the compactness of multipliers as well as on mapping properties with respect to Schatten classes. We also investigate perturbation results and the continuous dependence of the multiplier on the symbol and on the analysis and synthesis frames. 
We also look at the particular instances of STFT and Wavelet multipliers, the latter are known as Calder\'on-Toeplitz operators, and compare our results to existing ones.
Section 4 generalizes the concepts of controlled and weighted frames to the continuous setting.

\section{Preliminaries}

\subsection{Operator theory and functional analysis}

Throughout this paper, $\h$ (respectively $\h_1, \h_2$) will be
complex Hilbert spaces, with inner product $\langle x, y\rangle$,
linear in the first and conjugate linear in the second coordinate
and norm $\|x\| = \sqrt{\langle x, x  \rangle}$ for $x,y \in \h$.
 Let $\mathcal{B}(\h_1,\h_2)$ be the set of all bounded
linear operators from $\h_1$ to $\h_2$. 
This set is a Banach space with
norm $\|T\|=\sup_{\|x\|=1}\|Tx\|$. We define $GL(\h_1,\h_2)$ as the set
of all bounded linear operators with bounded inverse. If $\h_1 = \h_2 = \h$, we simply write $\mathcal{B}(\h)$ and $GL(\h)$. 
By $(e_n)$ we always denote an orthonormal basis for a Hilbert space. 
A map $\Psi:\h \times
\h\rightarrow \mathbb{C}$ is a sesquilinear form if it is linear in
the first variable and conjugate-linear in the second. For such a
map, we have the following assertion.
\begin{thm}\cite{Mu}\label{murphy}
 Let $\Psi$  be a bounded sesquilinear form on a Hilbert space $\h$.
Then there is a unique operator $u$ on $\h$ such that
\[
\Psi(x,y)=\langle u(x),y\rangle \quad (x,y\in \h). \]
 Moreover,
$\|u\|=\|\Psi\|$.
\end{thm}
A bounded operator $T$ is called positive (respectively
non-negative), if $\langle Tf,f\rangle>0$ for all $f\neq0$ (
respectively $\langle Tf,f\rangle\geq0$ for all $f\in\h$). We say
$S> T$ if $S-T>0$ (respectively $S\geq T$, if $S-T\geq0$). For a
non-negative operator $T$, there exists a unique non-negative
operator $S$ on $\h$ such that $S^2=T$ and $S$ commutes with every
operator that commutes with $T$. See e.g. \cite{conw1} or \cite{gogoka03} for good accounts of elementary operator theory.

A linear operator $T$ from the Banach space $X$ into the Banach space $Y$ is called
compact if the image under $T$ of the closed unit ball in $X$ is a relatively compact subset of $Y$, or, equivalently, if the image of any bounded sequence contains a convergent subsequence. A well-known characterization of compact operators is the following:
\begin{lem} \label{weakly}\cite{conw1}
Let $X,Y$ be Banach spaces. A bounded operator $T:X\rightarrow Y$
is compact if and only if $\|Tx_{n}\|\longrightarrow 0$ whenever
$x_{n}\longrightarrow 0$ weakly in $X$.
\end{lem}
For any compact operator $T:\h \to \K$, the operator $T^{\ast}T:\h \to \h$ is compact and non-negative. The unique non-negative operator $S$ such that $S^2 = T^{\ast}T$ is also compact. The eigenvalues of $S$ are called the singular values of $T$. They form a non-increasing sequence of non-negative numbers that either consists of only finitely many non-zero terms or converges to zero. If the sequence of singular values $(s_n)$ is in $\ell^p$, $1 \leq p < \infty$, then $T$ belongs to the Schatten p-class $\mathcal{S}_p(\h)$. In particular, if $\sum |s_n | < \infty$, then $T$ is a trace class operator; if $\sum | s_n |^2 < \infty$, then $T$ is a Hilbert-Schmidt operator. A good source for information on Schatten class operators is e.g. \cite{zh07}.
\\

We recall the definition of a discrete frame.
\begin{defn}
A family $(f_n)\subseteq\h$ is a \emph{frame} for $\h$ if there exist
constants $A>0$ and $B<\infty$ such that \[
A\|f\|^2\leq\sum_n|\langle f,f_n\rangle|^2\leq B\|f\|^2
\] for all $f\in\h$.
If $A=B$, then it is called a tight frame.

\end{defn}

\subsection{Continuous frames}

\begin{defn}
Let $\h $ be a complex Hilbert space and $(\Omega ,\mu)$ be a
measure space with positive measure $\mu.$ The mapping
$F:\Omega\to\h$ is called a \emph{continuous
frame} with respect to $(\Omega ,\mu)$, if\\
\begin{enumerate}
\item  $F$ is weakly-measurable, i.e., for all $f\in \h$,
$\omega\to\langle f,F({\omega})\rangle$
is a measurable function on $\Omega$; \\
\item  there exist constants $A, B> 0$ such that
\[
\label{deframe}
A\|f\|^{2}\leq \int_{\Omega}|\langle
f,F({\omega})\rangle|^{2}\,d\mu(\omega) \leq B\|f\|^{2}, \quad (
f\in \h).
\]
\end{enumerate}
The constants $A$ and $B$ are called \emph{continuous
frame bounds}. If $A=B$, then $F$ is called a \emph{tight} continuous frame, if $A = B = 1$ a {\em Parseval} frame. 
The
mapping $F$ is called {\em Bessel mapping} or shorter \emph{Bessel} if only the righthand inequality in
(\ref{deframe}) holds. In this case, $B$ is called the \emph{Bessel
constant} or \emph{Bessel bound}.
\end{defn}
If $\Omega=\mathbb{N}$ and $\mu$ is counting measure then $F$ is a
discrete frame. In this sense continuous frames are the more general
setting.

The first inequality in (\ref{deframe}), shows that $F$ is
complete, i.e.,
$$\overline{\textrm{span}}\{F({\omega})\}_{\omega\in\Omega}=\h.$$

It is well-known that discrete Bessel sequences in a Hilbert space are norm bounded above: if
\[
\sum_{n} | \langle f, f_n \rangle |^2 \leq B \| f \|^2
\]
for all $f\in \h$, then
\[
\| f_n \| \leq \sqrt{B}
\]
for all $n$. For continuous Bessel mappings, however, this is not necessary. Consider the following example.
\begin{exa}\label{example_1}
Take an (essentially) unbounded (Lebesgue) measurable function
$a:\Real \to \Comp$ such that $a \in L^2(\Real) \setminus
L^{\infty}(\Real)$. It is easy to see that such functions indeed
exist; consider for example the function
\[
b(x) := 
\begin{cases}
\frac{1}{\sqrt{|x|}}, & \mbox{ if } 0 < |x| < 1, \\
\frac{1}{|x|^2}, & \mbox{ if } |x| \geq 1, \\
0, & \mbox{ if } x=0.
\end{cases}
\]
This function is clearly in $L^1(\Real) \setminus L^{\infty}(\Real)$ and, furthermore, $b(x) \geq 0$ for all $x \in \Real$. Now take $a(x) := \sqrt{b(x)}$. Choose a fixed vector $h \in\h$, $h \neq 0$. Then the
mapping
\[
F: \Real \to\h, \quad \quad \omega \mapsto F(\omega) :=
a(\omega)\cdot h
\]
is weakly (Lebesgue) measurable and a continuous Bessel mapping, since
\begin{align*}
\int_{\Real} | \langle f, F(\omega) \rangle |^2\,d\mu(\omega) & = \int_{\Real} |a(\omega)|^2 | \langle f, h \rangle |^2\,d\mu(\omega) \\
& = | \langle f, h \rangle |^2 \int_{\Real} |a(\omega)|^2\,d\mu(\omega) \\
& \leq \| h \|^2 \|a\|_{L^2(\Real)}^2 \| f \|^2
\end{align*}
for all $f\in \h$, but
\[
\| F(\omega) \| = \| a(\omega) h \| = |a(\omega)| \| h \|
\]
is unbounded, since $a$ is unbounded.\hfill$\triangle$
\end{exa}

Even continuous frames need not necessarily be norm bounded.

\begin{exa}\label{example_2}
Let $F:\Real \to \h$ be a norm unbounded continuous Bessel mapping with Bessel constant $B_F$, as in the previous example. Let $G:\Real \to\h$ be a norm bounded continuous frame (for example a continuous wavelet or Gabor frame, cf. Section \ref{gaborandwavelet}) with continuous frame bounds $0 < A_G \leq B_G$ and norm bound $M > 0$, i.e. $\| G(\omega) \|\leq M$ for a.e. $\omega \in\Real$.\\
Then $G + \varepsilon F$ is a norm unbounded continuous frame, for all sufficiently small $\varepsilon > 0$.\\
To see this, first note that it is obvious that the mapping $G +
\varepsilon F:\Real \to\h$ is weakly measurable for any choice of
$\varepsilon > 0$. It satisfies the upper frame bound, since
\begin{align*}
\int_{\Real} | \langle f, G(\omega) & + \varepsilon F(\omega) \rangle |^2\,d\mu(\omega) \\
& \leq \int_{\Real} \Big( | \langle f, G(\omega) \rangle | + \varepsilon | \langle f, F(\omega) \rangle | \Big)^2 \,d\mu(\omega) \\
& \leq 2\cdot \int_{\Real} \Big( | \langle f, G(\omega) \rangle |^2 + \varepsilon^2 | \langle f, F(\omega) \rangle |^2 \Big) \,d\mu(\omega) \\
& \leq 2\cdot (B_G + \varepsilon^2 B_F)\cdot \| f \|^2.
\end{align*}

For the lower frame bound, observe that
\begin{align*}
\Big( \int_{\Real} | \langle f, G(\omega) & + \varepsilon F(\omega) \rangle |^2\,d\mu(\omega) \Big)^{1/2} \\
& \geq \Big( \int_{\Real} | \langle f, G(\omega) \rangle |^2\,d\mu(\omega) \Big)^{1/2} - \Big( \int_{\Real} \varepsilon^2 | \langle f, F(\omega) \rangle |^2\,d\mu(\omega) \Big)^{1/2} \\
& \geq \sqrt{A_G} \|f\| - \varepsilon \sqrt{B_F}\|f\| \\
& = (\sqrt{A_G} - \varepsilon\sqrt{B_F})\|f\|.
\end{align*}
Now choose $\varepsilon < \sqrt{\frac{A_G}{B_F}}$, then $\sqrt{A_G} - \varepsilon\sqrt{B_F} > 0$, and the lower frame bound is established.

This continuous frame is, however, not norm bounded, since
\begin{align*}
\| G(\omega) + \varepsilon F(\omega) \| \geq \varepsilon \| F(\omega) \| - \| G(\omega) \| \geq \varepsilon \| F(\omega) \| - M;
\end{align*}
by $F$ being unbounded, this is unbounded as well.\hfill$\triangle$
\end{exa}

The construction in the last two examples depends crucially on the existence of an unbounded square-integrable function or, equivalently, on the existence of an unbounded integrable function. It can be generalized to the following theorem:

\begin{thm}
Let $(\Omega, \mu)$ be a measure space such that $L^1(\Omega, \mu) \nsubseteq L^{\infty}(\Omega, \mu)$, i.e. there exist unbounded integrable functions. Then the following holds:\\
If there exist any continuous frames at all with respect to $(\Omega, \mu)$, then there are also norm-unbounded ones.
\end{thm}
\begin{proof}
Fix a vector $h \in \h$, $h \neq 0$. Pick a function $b:\Omega \to \mathbb{C}$, $b \in L^1(\Omega, \mu) \setminus L^{\infty}(\Omega, \mu)$. Then $a := \sqrt{| b|}$ is a function in $L^2(\Omega,\mu) \setminus L^{\infty}(\Omega, \mu)$, i.e. an unbounded square-integrable function. As in Example \ref{example_1}, the mapping $F:\Omega \to \h$, $F(\omega) = a(\omega)\cdot h$, is an norm-unbounded continuous Bessel mapping. If there exists a norm-bounded continuous frame $G:\Omega \to \h$, then one can show as in Example \ref{example_2} that the mapping $G + \epsilon F$ is a norm-unbounded continuous frame, for sufficiently small $\epsilon$. 
\end{proof}

Concerning the existence of continuous frames, we have the following result:

\begin{thm}
Let $(\Omega, \mu)$ be a $\sigma$-finite measure space. Then there exists a continuous tight frame $F:\Omega \to \h$ with respect to $(\Omega, \mu)$.
\end{thm}
\begin{proof}
Since $\Omega$ is $\sigma$-finite, it can be written as a disjoint union $\Omega = \bigcup \Omega_k$ of countably many subsets $\Omega_k \subseteq \Omega$ such that $\mu(\Omega_k) < \infty$ for all $k$. Without loss of generality assume that $\mu(\Omega_k) > 0$ for all $k$. If there are infinitely many such subsets $\Omega_k$, $k \in \mathbb{N}$, then let $(e_k)_{k\in\mathbb{N}}$ be an orthonormal basis of an infinite-dimensional separable Hilbert space $\h$. Define the function $F:\Omega \to \h$ by
\[
\omega \mapsto F(\omega) := \frac{1}{\sqrt{\mu(\Omega_k)}} e_k, \quad \mbox{ for }\omega \in \Omega_k.
\]
Then, for all $f \in \h$,
\begin{align*}
\int_{\Omega} \left| \langle f, F(\omega) \rangle \right|^2\,d\mu(\omega) & = \sum_{k} \int_{\Omega_k}\left| \langle f, F(\omega) \rangle \right|^2\,d\mu(\omega) \\
& = \sum_k \left| \langle f, e_k \rangle \right|^2 \frac{1}{\mu(\Omega_k)} \mu(\Omega_k) \\
& = \| f \|^2,
\end{align*}
thus $F$ is a continuous tight frame with frame bound $1$. If there are only finitely many $\Omega_k$, $k = 1, \ldots, N$, then take for $\h$ an $N$-dimensional Hilbert space instead and proceed analogously.
\end{proof}

For the convenience of the reader, we shortly repeat some basic facts and notions on continuous frames. Details may be found for example in \cite{Ali2} or \cite{RaNaDe}.

Let $F$ be a continuous frame with respect to $(\Omega ,\mu)$,
then the mapping
$$\Psi : \h\times\h \to \mathbb{C}$$ defined by
$$\Psi(f,g)= \int_{\Omega}\langle f,F({\omega})\rangle\langle
F({\omega}),g\rangle \,d\mu(\omega)$$ is well defined, sesquilinear
and bounded. By Cauchy-Schwarz's inequality we get
\begin{eqnarray*}|\Psi(f,g)|&\leq&
\int_{\Omega}|\langle f,F(\omega)\rangle\langle
F(\omega),g\rangle|\, d\mu(\omega)\\
&\leq& \left(\int_{\Omega}|\langle f,F(\omega)\rangle|^{2}
\,d\mu(\omega)\right)^{\frac{1}{2}}\left(\int_{\Omega}|\langle
F(\omega), g \rangle|^{2}\, d\mu(\omega)\right)^{\frac{1}{2}}\\
&\leq& B\| f\|\| g\|.
\end{eqnarray*}

Hence $\|\Psi\|\leq B.$ By Theorem \ref{murphy} there exists a
unique operator $S_{F}: \h\to\h$ such that
$$\Psi(f,g)=\langle S_{F}f, g\rangle,\quad ( f,g\in\h )$$ and moreover
$\|\Psi\|=\|S\|.$\\
Since $\langle S_{F}f,f\rangle=\int_{\Omega}|\langle
f,F(\omega)\rangle |^{2}\,d\mu(\omega)$, $S_{F}$ is positive and
$AI\leq S_{F}\leq BI$. Hence $S_{F}$ is invertible, positive and $\frac{1}{B} I\leq S^{-1}_{F}\leq \frac{1}{A} I$.
We call
$S_{F}$ the continuous frame operator of $F$ and we use the
notation $S_{F}f=\int_{\Omega}\langle f,F(\omega)\rangle F(\omega)
\,d\mu(\omega)$, which is valid in the weak sense.
Thus, every $f\in\h$ has the (weak) representations
$$f=S_{F}^{-1}S_{F}f=\int_{\Omega}\langle f, F(\omega)\rangle
S_{F}^{-1}F(\omega)\,d\mu(\omega)$$$$f=S_{F}S_{F}^{-1}f=\int_{\Omega}\langle
f, S_{F}^{-1}F(\omega)\rangle F(\omega)\,d\mu(\omega).$$
\begin{thm}\label{TF}\cite{RaNaDe}
Let $(\Omega, \mu)$ be a measure space and let $F$ be a Bessel
mapping from $\Omega$ to $\h.$ Then the operator
$T_{F}:L^{2}(\Omega, \mu)\to\h$ weakly defined  by $$\langle
T_{F}\varphi, h\rangle=\int_{\Omega}\varphi(\omega)\langle
F(\omega),h \rangle \,d\mu(\omega),\quad ( h\in\h )$$ is well
defined, linear, bounded and its adjoint is given by $$ T_{F}^{*}:
\h\to L^{2}(\Omega, \mu),\quad (T_{F}^{*}h)(\omega)=\langle h,
F(\omega)\rangle,\quad ( \omega\in\Omega ).$$ The operator $T_{F}$
is called the \textit{synthesis operator} and
$T_{F}^{*}$ is called the \textit{analysis operator} of $F$.
\end{thm}

Such as in the discrete case we have the next proposition.
\begin{prop}\cite{RaNaDe}
Let $F:\Omega\to\h$ be a Bessel function with respect to
$(\Omega,\mu)$. By the above notations $S_{F}=T_{F}T_{F}^{*}.$
\end{prop}

Using an analogous statement as in \cite{RaNaDe} for the synthesis operator, it is easy to prove a characterization of continuous frames in terms of the frame operator.
\begin{thm}
Let $(\Omega, \mu)$ be a $\sigma$-finite measure space.

The mapping $F:\Omega\to\h$ is a continuous
frame with respect to $(\Omega, \mu)$ for $\h$ if and only if the
operator $S_{F}$ is a bounded and invertible operator.
\end{thm}

\begin{defn}
Let $F$ and $G$ be continuous frames with respect to $(\Omega
,\mu)$ for $\h$. We call $G$ a \textit{dual} of $F$ if the
following holds true:
\[
\label{dual}
\langle f,g\rangle=\int_{\Omega}\langle
f,F(\omega)\rangle\langle G(\omega),g\rangle d\mu \quad ( f, g\in\h ).
\]
In this case $(F,G)$ is called a \textit{dual pair}. It is clear
that (\ref{dual}) is equivalent with $T_G T^*_F=I$.
\end{defn}

It is certainly possible for a continuous frame $F$ to have only one dual.
In this case we call $F$ a \textit{Riesz-type} frame.

\begin{prop}\cite{Gb} Let $F$ be a continuous frame with respect to $(\Omega
,\mu)$ for $\h$. Then $F$ is a Riesz-type frame if and only if
$\R(T_{F}^* ) = L^2(\Omega , \mu)$.

\end{prop}

\subsection{Gabor and wavelet systems}\label{gaborandwavelet}

Well known examples for frames are wavelet and Gabor systems. The corresponding continuous wavelet and STFT transforms give rise to continuous frames. We make use of the following unitary operators on $L^2(\mathbb{R})$:
\begin{itemize}
\item Translation: $T_x f(t) := f(t - x)$, for $f \in L^2(\mathbb{R})$ and $x \in \mathbb{R}$
\item Modulation: $M_y f(t) := e^{2\pi i y\cdot t} f(t)$, for $f \in L^2(\mathbb{R})$ and $y \in \mathbb{R}$
\item Dilation: $D_z f(t) := \frac{1}{|z|^{\frac{1}{2}}} f(\frac{t}{z})$, for $f \in L^2(\mathbb{R})$ and $z > 0$
\end{itemize}
\begin{defn}\label{D:Def_Wavelet}
Let $ \psi \in L^{2}(\mathbb{R})$, and let
$$C_{\psi}:=\int_{-\infty}^{+\infty}\frac{|\hat{\psi}(\gamma)|^{2}}
{|\gamma|}\,d\gamma,$$
where $\hat{\psi}$ denotes the Fourier transform of $\psi$. The function $\psi$ is called admissible if $0 < C_{\psi} < +\infty$. For $a,b\in\mathbb{R}$ with
$a\neq0$, let
$$\psi^{a,b}(x):= (T_{b}D_{a}\psi)(x)= \frac{1}{|
a|^{\frac{1}{2}}}\psi(\frac{x-b}{a}), \quad ( x\in\mathbb{R}).$$
Then the \textit{continuous wavelet transform} $W_{\psi} $ is
 defined by
$$W_{\psi}(f)(a,b):=\langle
f,\psi^{a,b}\rangle=\int_{-\infty}^{+\infty}f(x)\frac{1}{|
a|^{\frac{1}{2}}}\overline{\psi(\frac{x-b}{a})}\,dx, \quad f \in L^2(\mathbb{R}). $$
\end{defn}

For an admissible function $\psi$ in $L^2$, the system $\{\psi^{a,b}\}_{a\neq0, b\in\mathbb{R}}$
is a continuous tight frame for $ L^{2}(\mathbb{R})$ with respect to $
\Omega = \mathbb{R}\setminus\{0\}\times\mathbb{R} $ equipped with the
measure $\frac{dadb}{a^{2}}$ and for all $f\in L^{2}(\mathbb{R})$
$$f=\frac{1}{C_{\psi}}\int_{-\infty}^{+\infty}\int_{-\infty}^{+\infty}W_{\psi}(f)(a,b)
\psi^{a,b}\,\frac{dadb}{a^{2}},$$
where the integral is understood in weak sense (this formula is known as the \textit{Calder\'on Reproducing Formula}, cf. \cite{da92}). This system constitutes a continuous tight frame with frame bound $\frac{1}{C_{\psi}}$. If $\psi$ is suitably normed so that $C_{\psi} = 1$, then the frame bound is $1$, i.e. we have a continuous Parseval frame.
For details, see the Proposition 11.1.1 and Corollary 11.1.2 of
\cite{C4}.

\begin{defn}\label{D:Def_STFT}
Fix a function $ g\in L^{2}(\mathbb{R})\setminus\{0\}$. The
\textit{short-time Fourier transform} (STFT) of a function $f\in
L^{2}(\mathbb{R})$ with respect to the window function $g $ is given
by $$
\Psi_{g}(f)(y,\gamma)=\int_{-\infty}^{+\infty}f(x)\overline{g(x-y)}e^{-2\pi
i x\gamma}dx, \quad\quad( y, \gamma \in\mathbb{R}).$$
\end{defn}
Note that in terms of modulation operators and translation operators,
$\Psi_{g}(f)(y,\gamma)=\langle f, M_{\gamma} T_{y}g\rangle$.

Let $ g\in L^{2}(\mathbb{R})\setminus\{0\}$. Then
$\{M_{b}T_{a}g\}_{a,b\in\mathbb{R}}$is a continuous frame for
$L^{2}(\mathbb{R})$ with respect to $ \Omega =\mathbb{R}^{2}$ equipped
with the Lebesgue measure. Let $f_{1}, f_{2}, g_{1}, g_{2}\in
L^{2}(\mathbb{R})$. Then
$$\int_{-\infty}^{+\infty}\int_{-\infty}^{+\infty}\Psi_{g_{1}}(f_{1})(a,b)\overline{\Psi_{g_{2}}(f_{2})(a,b)}dbda
=\langle f_{1}, f_{2}\rangle\langle g_{2},g_{1}\rangle.$$
So this system represent a continuous tight frame with bound $\|g\|^2$. For
details see the proposition 8.1.2 of \cite{C4}.

\section{Continuous Frame Multipliers}

Gabor multipliers \cite{feic} led to the introduction of
Bessel and frame multipliers for abstract Hilbert spaces
$\mathcal{H}_{1}$ and $\mathcal{H}_{2}$. These operators are
defined by a fixed multiplication pattern (the symbol) which is
inserted between the analysis and synthesis operators
\cite{xxlmult1,peter2,peter3}.

\begin{defn}
Let $\mathcal{H}_{1}$ and $\mathcal{H}_{2}$ be Hilbert spaces, let
$(\psi_{k})\subseteq\mathcal{H}_{1}$ and
$(\phi_{k})\subseteq\mathcal{H}_{2}$ be Bessel sequences. Fix
$m = (m_k)\in l^{\infty}$. The operator ${\bf M}_{(m_k), (\phi_k), (\psi_k)}
: \mathcal{H}_{1} \rightarrow \mathcal{H}_{2}$ defined by
$$ {\bf M}_{(m_k), (\phi_k), ( \psi_k )} (f)  =  \sum \limits_k m_k
\langle f,\psi_k\rangle \phi_k,\quad \quad (f \in \mathcal{H}_1)  $$ is called \emph{Bessel multiplier}
for the Bessel sequences $\{\psi_{k}\}$ and $\{\phi_{k}\}$. The
sequence $m$ is called the \emph{symbol} of {\bf M}. For frames the resulting Bessel multiplier is called a \emph{frame multiplier}, for Riesz
sequence a \emph{Riesz multiplier}.
\end{defn}

This motivates the following definition in the continuous case.

\begin{defn}\label{definitioncontframemult}
Let $F$ and $G$ be Bessel mappings for $\h$ with respect
to $(\Omega,\mu)$ and $m:\Omega\rightarrow \mathbb{C}$ be a
measurable function. The operator
$\mathbf{M}_{m,F,G}:\h\rightarrow\h$ weakly defined by
\[
\langle
\mathbf{M}_{m,F,G}f,g\rangle=\int_{\Omega}m(\omega)\langle f,
F(\omega)\rangle\langle G(\omega),g\rangle d\mu(\omega)
\]
for all $f,g\in\h$, is called {\em continuous Bessel multiplier} of $F$
and $G$ with respect to the mapping $m$, called the  {\em symbol}.
\end{defn}
We use the following notation to be understood in weak sense as above:
$$\mathbf{M}_{m,F,G}f:=\int_{\Omega}m(\omega)\langle f,
F(\omega)\rangle G(\omega)d\mu(\omega).$$

\begin{lem}\label{tar}
Let $F$ and $G$ be Bessel mappings for $\h$ with respect
to $(\Omega,\mu)$ with bounds $B_F$ and $B_G$. Let $m\in
L^{\infty}(\Omega,\mu)$. The operator
$\mathbf{M}_{m,F,G}:\h\rightarrow\h$ weakly defined by
\[
\langle
\mathbf{M}_{m,F,G}f,g\rangle=\int_{\Omega}m(\omega)\langle f,
F(\omega)\rangle\langle G(\omega),g\rangle d\mu(\omega)
\]
for all $f,g\in\h$, is well defined and bounded with
\[
\|\mathbf{M}_{m,F,G}\|\leq \|m\|_\infty \sqrt{B_F B_G}.
\]

\end{lem}
\begin{proof}
It is clear that for each $f,g\in\h$,
\begin{eqnarray*}
|\langle\mathbf{M}_{m,F,G}f,g\rangle|&\leq&\|m\|_\infty
\int_{\Omega}|\langle f,F(\omega)\rangle\langle
G(\omega),g\rangle|\, d\mu(\omega)\\
&\leq&\|m\|_\infty \left(\int_{\Omega}|\langle
f,F(\omega)\rangle|^{2}
\,d\mu(\omega)\right)^{\frac{1}{2}}\left(\int_{\Omega}|\langle
G(\omega), g \rangle|^{2}\, d\mu(\omega)\right)^{\frac{1}{2}}\\
&\leq& \|m\|_\infty \sqrt{B_F B_G}\| f\|\| g\|.
\end{eqnarray*}
Thus $\mathbf{M}_{m,F,G}$ is well defined and bounded.
\end{proof}

It is easy to prove that if $m(\omega)>0$ a.e., then for any Bessel
function $F$ the multiplier $\mathbf{M}_{m,F,F}$ is a positive
operator, and if $m(\omega)\geq \delta > 0$ for some positive constant $\delta$ and
$\|m\|_{\infty}<\infty$, then $\mathbf{M}_{m,F,F}$ is just the
frame operator of $\sqrt{m}F$ and thus is positive, self-adjoint
and invertible.
\par
By using synthesis and analysis operators, one easily shows that
\begin{equation}\label{rep1}\mathbf{M}_{m,F,G}=T_G D_m T^*_F
\end{equation}

where $D_m:L^{2}(\Omega,\mu)\rightarrow L^{2}(\Omega,\mu) $ and
$(D_m \varphi)(\omega)=m(\omega)\varphi(\omega)$. It is proved that
if $m\in L^{\infty}(\Omega,\mu)$, then $D_m$ is bounded and
$\|D_m\|=\|m\|_{\infty}$, \cite{conw1}.
\par

\begin{prop}
Let $F$ and $G$ be Bessel mappings for $\h$ with respect
to $(\Omega,\mu)$ and $m:\Omega\rightarrow \mathbb{C}$ be a
measurable function, then
 $(\mathbf{M}_{m,F,G})^*=\mathbf{M}_{\overline{m},G,F}.$
\end{prop}
\begin{proof}
 For $f,g\in\h$
\begin{eqnarray*}
\langle f,\mathbf{M}^*_{m,F,G}g\rangle
&=&\langle\mathbf{M}_{m,F,G}f,g\rangle\\
&=&\int_{\Omega}m(\omega)\langle f, F(\omega)\rangle\langle
G(\omega),g\rangle d\mu(\omega)\\
&=&\int_{\Omega}\langle f,\overline{m(\omega)}\langle
g,G(\omega)\rangle F(\omega)\rangle d\mu(\omega)\\
&=&\langle f,\mathbf{M}_{\overline{m},G,F}g\rangle.
\end{eqnarray*}
\end{proof}

\subsection{Multiplication operators on $L^2$}

Motivated by the discrete case one might expect that $m\in
L^p$ implies $D_m\in \mathcal{S}_p$, where $\mathcal{S}_p(\h)$
denotes the family of Schatten $p$-class operators on $\h$. For
$p=1$, we have trace class operators and for $p=2$ we have
Hilbert-Schmidt operators.
If this were true, we could easily, using the representation (\ref{rep1}), get results like in \cite{peter2},
since $ \mathcal{S}_p(\h_1,\h_2) $ is a
two sided $*$-ideal of $\mathcal{B}(\h_1, \h_2)$.
Unfortunately, the following proposition shows that at least for multiplication operators on $L^2(\RR^d) = L^2(\RR^d, dx)$ (with $dx$ denoting Lebesgue measure) the
above considerations are never true, which constitutes a major difference between the discrete and the continuous case. The result seems to be mathematical folklore; we give a full proof for completeness.
We will use the following lemma, which is of independent interest.
\begin{lem}
Let $A \subseteq \RR^d$ be a measurable set of positive Lebesgue measure, $\lambda(A) > 0$. Then there exists a partition of $A$ into countably infinitely many mutually disjoint measurable sets $A_n$, $n\in\NN$, of positive measure, i.e. such that
\begin{enumerate}
\item $A = \bigcup_{n=1}^{\infty} A_n$,
\item $A_n \cap A_m = \emptyset$ for $n\not=m$,
\item $\lambda(A_n) > 0$ for all $n\in \NN$.
\end{enumerate}
\end{lem}
\begin{proof}
It suffices to show that any $A\subseteq \RR^d$ with $\lambda(A) > 0$ can be decomposed in \emph{two} disjoint measurable sets $B$ and $C$ such that $A = B \cup C$, $B \cap C = \emptyset$ and $\lambda(B) > 0$, $\lambda(C) > 0$, since the claim follows from this by induction. Without loss of generality assume further that $\lambda(A) =: L < \infty$. The whole space $\RR^d$ can be covered by mutually disjoint d-dimensional half-open cubes $I_n$, $n\in\NN$, sufficiently small such that $\lambda(I_n) < \frac{L}{2}$ for all $n\in\NN$. Then $A = \bigcup_{n\in\NN} (A\cap I_n)$. Since $(A\cap I_n) \cap (A\cap I_m) = \emptyset$, we have $L = \lambda(A) = \sum_{n\in\NN} \lambda(A\cap I_n)$. Now set
\[
N := \{ n\in\NN\,:\, \lambda(A\cap I_n) > 0 \}.
\]
Since $\lambda(A\cap I_n) \leq \lambda(I_n) \leq \frac{L}{2}$ for all $n\in\NN$, $N$ must clearly contain \emph{at least two elements}, say $n_1$ and $n_2$. Now set
\[
B := A\cap(I_{n_1})
\]
and
\[
C := (A\cap I_{n_2}) \cup \bigcup_{n\not\in N} (A\cap I_n).
\]
Then $B$ and $C$ have the stated properties.
\end{proof}
Now we can prove the following
\begin{prop}
Let $a\in L^{\infty}(\RR^d)$. Denote by $D_a:L^2(\RR^d) \to L^2(\RR^d)$, $f \mapsto a\cdot f$, the bounded multiplication operator with symbol $a$.\\
Then $D_a$ is a compact operator if and only if $a = 0$.
\end{prop}
\begin{proof}
Assume $a \not= 0$. Let $\| a \|_{\infty} =: c > 0$. Define
\[
A := \{ x\in\RR^d\, : \, |a(x)| > \frac{c}{2} \}.
\]
Then $A$ is a set of positive Lebesgue measure, $\lambda(A) > 0$. Find a partition of $A$ as in the preceding lemma, i.e. into countably infinitely many measurable subsets $A_n$, $n\in\NN$, such that (1) $A = \bigcup_{n\in\NN} A_n$, (2) $A_n \cap A_m = \emptyset$ for $n\not= m$, i.e. the sets $A_n$ are mutually disjoint, and (3) $\lambda(A_n) > 0$ for all $n\in\NN$, i.e. all the sets $A_n$ have strictly positive Lebesgue measure. Then set
\[
f_n := \chi_{A_n}\cdot \frac{1}{\sqrt{\lambda(A_n)}},
\]
with $\chi_{A_n}$ the characteristic function of $A_n$.
Since
\begin{align*}
\langle f_n, f_m \rangle & = \int_{\RR^d} \chi_{A_n}(x)\cdot \frac{1}{\sqrt{\lambda(A_n)}} \cdot \overline{\chi_{A_m}(x)\cdot \frac{1}{\sqrt{\lambda(A_m)}}}\,dx \\
& = \int_{A_n\cap A_m} \frac{1}{\sqrt{\lambda(A_n)\lambda(A_m)}} \\
& = \begin{cases}
\int_{A_n} \frac{1}{\lambda(A_n)}\,dx = 1, & \text{ if $n = m$,} \\
0, & \text{ if $n\not= m$,}
\end{cases}
\end{align*}
the sequence of functions $(f_n)$ constitutes an orthonormal system in $L^2(\RR^d)$. As such, it satisfies $f_n \to 0$ weakly by Bessel's Inequality. But $D_af_n(x) = a(x)\cdot \chi_{A_n}(x)\frac{1}{\sqrt{\lambda(A_n)}}$ and
\begin{align*}
\| D_af_n \|^2 & = \int_{\RR^d} |a(x)|^2\cdot |\chi_{A_n}(x)\frac{1}{\sqrt{\lambda(A_n)}}|^2\,dx \\
& = \int_{A_n} |a(x)|^2 \frac{1}{\lambda(A_n)}\,dx \\
& \geq \int_{A_n} \big(\frac{c}{2}\big)^2 \frac{1}{\lambda(A_n)}\,dx \\
& \geq \big(\frac{c}{2}\big)^2,
\end{align*}
thus $\| D_af_n \| \nrightarrow 0$. Hence $D_a$ cannot be compact by Lemma \ref{weakly}.
\end{proof}

In order to prove sufficient conditions for compactness of continuous frame multipliers, we thus have to choose a different approach than in the discrete setting. This will be addressed in the next section.

\subsection{Compact Multipliers}

\begin{thm}
Let $F$ and $G$ be Bessel mappings for $\h$ with respect to $(\Omega, \mu)$ and let either $F$ or $G$ be norm bounded, i.e. there is a constant $M > 0$ such that $\| F(\omega) \| \le M$ resp. $\| G(\omega) \| \le M$ for almost every $\omega \in \Omega$. Let $m:\Omega\rightarrow \mathbb{C}$ be a (essentially) bounded measurable function with support of finite measure, i.e. there exists a subset $K \subseteq \Omega$ with $\mu(K) < \infty$ such that $m(\omega) = 0$ for almost every $\omega \in \Omega \setminus K$.\\
Then $\mathbf{M}_{m,F,G}$ is a compact operator.
\end{thm}
\begin{proof}
We have
\[
\mathbf{M}_{m,F,G} = T_G \circ D_m \circ T^{\ast}_F
\]
with $T^{\ast}_F$ the analysis operator for $F$, $D_m$ the multiplication operator with symbol $m$ and $T_G$ the synthesis operator for $G$. Assume first that $F$ is bounded, $\|F(\omega)\| \le M$ for almost all $\omega \in \Omega$. We will show that $D_m \circ T^{\ast}_F: \h \to L^2(\Omega, \mu)$ is compact. To this end, let $f_n \to 0$ weakly. Then
\begin{align*}
\| D_mT^{\ast}_Ff_n \|^2 & = \int_{\Omega} |m(\omega)|^2\cdot | \left< f_n, F(\omega) \right> |^2\,d\mu(\omega) \\
& = \int_{K} |m(\omega)|^2\cdot | \left< f_n, F(\omega) \right> |^2\,d\mu(\omega).
\end{align*}
For the integrand,
\[
|m(\omega)|^2\cdot | \left< f_n, F(\omega) \right> |^2 \to 0
\]
for $n\to\infty$ pointwise for every fixed $\omega \in K$, since the weak convergence of $(f_n)$ implies $\left< f_n, F(\omega) \right> \to 0$ for every $\omega\in\Omega$ fixed. Furthermore, weakly convergent sequences are bounded, thus there is a constant $C>0$ such that $\|f_n \| \le C$ for all $n\in\mathbb{N}$, and
\begin{align*}
|m(\omega)|^2\cdot | \left< f_n, F(\omega) \right> |^2 & \le \|m\|_{\infty}^2\cdot \|f_n\|^2\cdot \|F(\omega)\|^2 \\
& \le \|m\|_{\infty}^2 \cdot C^2 \cdot M^2
\end{align*}
for all $n\in\mathbb{N}$. This constant is an integrable majorant on $K$, so by Lebesgue's Dominated Convergence Theorem
\[
\int_{K} |m(\omega)|^2\cdot | \left< f_n, F(\omega) \right> |^2\,d\mu(\omega) \to 0
\]
for $n\to \infty$. Hence the operator $D_m \circ T^{\ast}_F$ maps weakly convergent sequences to norm convergent ones and is compact by Lemma \ref{weakly}. So $\mathbf{M}_{m,F,G} = T_G \circ (D_m \circ T^{\ast}_F)$ is compact as well.\\
If $G$ is bounded instead of $F$, consider the adjoint operator
\[
\mathbf{M}_{m,F,G}^{\ast} = \mathbf{M}_{\overline{m},G,F} = T_{F}\circ D_{\overline{m}} \circ T^{\ast}_G;
\]
by what we have already shown, $\mathbf{M}_{\overline{m},G,F}$ is compact, hence also $\mathbf{M}_{m,F,G}$.
\end{proof}

\begin{cor} \label{sec:compact2}
Let $F$ and $G$ be Bessel mappings for $\h$ with respect to $(\Omega, \mu)$ and let either $F$ or $G$ be norm bounded. Let $m:\Omega \to \mathbb{C}$ be a (essentially) bounded measurable function that vanishes at infinity, i.e. for every $\varepsilon > 0$ there is a set of finite measure $K = K(\varepsilon) \subseteq \Omega$, $\mu(K) < \infty$, such that $m(\omega)\le \varepsilon$ for almost every $\omega \in \Omega \setminus K$. Then $\mathbf{M}_{m,F,G}$ is compact.
\end{cor}
\begin{proof}
For every $n\in\mathbb{N}$, choose a set $K_n \subseteq \Omega$ such that $\mu(K_n) < \infty$ and $|m(\omega)| \le \frac{1}{n}$ for all $\omega \not\in K_n$. Set
\[
m_n(\omega) := m(\omega)\cdot \chi_{K_n}(\omega)
\]
where $\chi_{K_n}$ denotes the characteristic function of the set $K_n$. Then obviously
\[
\| m_n - m \|_{\infty} \le \frac{1}{n} \to 0
\]
for $n\to \infty$, thus
\[
\| \mathbf{M}_{{m_n},F,G} - \mathbf{M}_{m,F,G}\| \le \| m_n - m \|_{\infty} \sqrt{B_FB_G} \to 0
\]
by Lemma \ref{tar}. The functions $m_n$ are bounded and of finite support, so $\mathbf{M}_{{m_n},F,G}$ is compact for every $n\in\mathbb{N}$ by the preceding theorem, hence $\mathbf{M}_{m,F,G}$ is also compact.
\end{proof}

Now assume that \emph{both} $F$ and $G$ are norm bounded. Then we can prove a trace class result. We use the following criterion:

\begin{lem}\label{trace_class_criterion}
\cite{pietsch} Let $\h$ be a Hilbert space. A bounded operator $T:\h \to \h$ is trace class if and only if $\sum_n | \left<Te_n, e_n \right> | < \infty$ for every orthonormal basis $(e_n)$ of $\h$. Moreover,
\[
\| T \|_{\mathcal{S}^1} = \sup \{\sum_n | \left<Te_n, e_n \right> |\,:\, (e_n) \mbox{ orthonormal basis } \}.
\]
\end{lem}

\begin{thm}\label{ubs}
Let $F$ and $G$ be norm bounded Bessel mappings with norm bounds $L_F$ and $L_G$, respectively. Let $m \in L^1(\Omega,\mu)$.\\
Then $\mathbf{M}_{m,F,G}$ is a well defined bounded operator and a trace class operator with $\| \mathbf{M}_{m,F,G} \|_{\mathcal{S}_1} \le \|m\|_{1}L_F L_G$.
\end{thm}
\begin{proof}
For arbitrary $f, g \in \h$, we have
\begin{align*}
\int_{\Omega} |m(\omega)|& | \left< f, F(\omega) \right> || \left< G(\omega), g \right> |\,d\mu(\omega) \\
& \leq \int_{\Omega} |m(\omega)| \|f\| \|F(\omega)\| \|g\| \|G(\omega)\|\,d\mu(\omega) \\
& \leq \|f\| \|g\| L_F L_G \int_{\Omega} |m(\omega)|\,d\mu(\omega) \\
& = \|f\| \|g\| L_F L_G \|m\|_1,
\end{align*}
thus $\mathbf{M}_{m,F,G}$ is a well defined bounded linear operator by Theorem \ref{murphy}.
\\
Take an arbitrary orthonormal basis $(e_n)$ of $\h$. Then
\begin{align*}
\sum_n & |\left< \mathbf{M}_{m,F,G}e_n, e_n \right>| \\
& = \sum_n |\int_{\Omega} m(\omega) \left< e_n, F(\omega) \right>\left< G(\omega), e_n \right>\,d\mu(\omega) | \\
& \leq \sum_n \int_{\Omega} | m(\omega)|\cdot | \left< e_n, F(\omega) \right>| \cdot |\left< G(\omega), e_n \right>  |\,d\mu(\omega) \\
& \stackrel{\mbox{\tiny Fub.}}{=} \int_{\Omega} |m(\omega)| \sum_n | \left< e_n, F(\omega) \right>| \cdot |\left< G(\omega), e_n \right>  |\,d\mu(\omega) \\
& \stackrel{\mbox{\tiny C.-S.}}{\leq} \int_{\Omega} |m(\omega)| \left( \sum_n | \left< e_n, F(\omega) \right>|^2 \right)^{1/2}\left( \sum_n |\left< G(\omega), e_n \right>  |^2 \right)^{1/2}\,d\mu(\omega) \\
& = \int_{\Omega} |m(\omega)| \|F(\omega) \|\|G(\omega) \|\,d\mu(\omega) \\
& \leq \|m\|_{1}L_F L_G,
\end{align*}
where we have used Fubini's Theorem and the Cauchy-Schwarz's Inequality at the indicated places.
Hence $\mathbf{M}_{m,F,G}$ is trace class with norm estimate $\| \mathbf{M}_{m,F,G} \|_{\mathcal{S}_1} \le \|m\|_{1}L_F L_G$, by the previous Lemma \ref{trace_class_criterion}.
\end{proof}

Having established the trace class case, we are now able to extend the result to the whole family of Schatten p-classes by complex interpolation, see e.g. \cite{belo}.

\begin{thm} \label{sec:schatten1}
Let $F$ and $G$ be norm bounded Bessel mappings with norm bounds $L_F$ and $L_G$, respectively. Let $m \in L^p(\Omega,\mu)$, $1 < p < \infty$.\\
Then $\mathbf{M}_{m,F,G}$ is a well defined bounded operator that belongs to the Schatten p-class $\mathcal{S}_p(\h)$, with norm estimate
\[
\| \mathbf{M}_{m,F,G} \|_{\mathcal{S}_p} \leq \| m \|_p \left( L_F L_G \right)^{1/p}\left(B_F B_G \right)^{1/2q}.
\]
\end{thm}
\begin{proof}
We first show that the operator is well defined by the weak definition in \ref{definitioncontframemult}. To this end, let $f,g \in \h$ be fixed. Observe that the functions $\omega \mapsto \left< f, F(\omega) \right>$ resp. $\omega \mapsto \left< G(\omega), g \right>$ are bounded (by $L_F \|f\|$ resp. $L_G \|g\|$) and belong to $L^2(\Omega, \mu)$ (because $F$ and $G$ are Bessel mappings), hence their product $\omega \mapsto \left< f, F(\omega) \right>\left< G(\omega), g \right>$ is in $L^1(\Omega, \mu) \cap L^{\infty}(\Omega, \mu)$. But $L^1(\Omega, \mu) \cap L^{\infty}(\Omega, \mu) \subseteq L^q(\Omega, \mu)$ for all $1<q<\infty$.

Thus, for all $f,g \in\h$,
\begin{align*}
| \left< \mathbf{M}_{m,F,G}f,g \right> | & \leq \int_{\Omega} |m(\omega)| | \left< f, F(\omega) \right>\left< G(\omega), g \right> |\,d\mu(\omega) \\
& \leq \| m \|_{p} \| \left< f, F(\cdot) \right>\left< G(\cdot), g \right> \|_{q}
\end{align*}
by H\"older's Inequality, with $\frac{1}{p} + \frac{1}{q} = 1$. The second term can be estimated as
\begin{align*}
\| \left< f, F(\cdot) \right>&\left< G(\cdot), g \right> \|_{q} \\
& \leq  \|\left< f, F(\cdot) \right>\left< G(\cdot), g \right> \|^{q-1}_{\infty} \| \left< f, F(\cdot) \right>\left< G(\cdot), g \right> \|_{1} \\
& \leq L_F^{q-1}\|f\|^{q-1} L_G^{q-1}\|g\|^{q-1} \int_{\Omega} | \left< f, F(\omega) \right>\left< G(\omega), g \right> | \,d\mu(\omega) \\
& \stackrel{\mbox{\tiny C.-S.}}{\leq} L_F^{q-1}\|f\|^{q-1} L_G^{q-1}\|g\|^{q-1} \sqrt{B_F} \,\|f\| \sqrt{B_G}\, \|g\| \\
& = L_F^{q-1} L_G^{q-1}\sqrt{B_F} \sqrt{B_G}\, \|f\|^q \|g\|^q.
\end{align*}
Now assume that $\|f\|, \|g\| \leq 1$. Then
\[
| \left< \mathbf{M}_{m,F,G}f,g \right> | \leq \| m \|_{p}(L_F L_G)^{q-1}\sqrt{B_F B_G},
\]
thus for arbitrary $f,g \in\h$
\[
| \left< \mathbf{M}_{m,F,G}f,g \right> | \leq \| m \|_{p}(L_F L_G)^{q-1}\sqrt{B_F B_G}\,\|f\| \|g\|.
\]
This proves that $\mathbf{M}_{m,F,G}$ is a well defined bounded operator.\\
Now Lemma \ref{tar} shows that the mapping $L^{\infty}(\Omega,\mu) \to \mathcal{B}(\h)$, $m \mapsto \mathbf{M}_{m,F,G}$, is a bounded linear operator. The same is true for the mapping $L^1(\Omega, \mu) \to \mathcal{S}_1(\h)$, $m \mapsto \mathbf{M}_{m,F,G}$, by Theorem \ref{ubs}. Now let $\theta = 1 - \frac{1}{p} = \frac{1}{q}$ (i.e. such that $\frac{1}{p} = \frac{1-\theta}{1} + \frac{\theta}{\infty}$). A standard complex interpolation (\cite{belo}), between the Banach spaces $[L^1(\Omega,\mu), L^{\infty}(\Omega,\mu) ]_{\theta} = L^p(\Omega,\mu)$ on the one hand and $[ \mathcal{S}_1(\h), \mathcal{B}(\h)] = [\mathcal{S}_1(\h), \mathcal{S}_{\infty}(\h) ]_{\theta} = \mathcal{S}_p(\h)$ on the other, proves that the mapping $m \mapsto \mathbf{M}_{m,F,G}$ gives also a bounded linear operator from $L^p(\Omega,\mu)$ to the Schatten p-class $\mathcal{S}_p(\h)$ with norm estimate
\begin{align*}
\| \mathbf{M}_{m,F,G} \|_{\mathcal{S}_p} & \leq \| m \|_p \left( L_F L_G \right)^{1 - \theta}\left( \sqrt{B_F B_G} \right)^{\theta} \\
& = \| m \|_p \left( L_F L_G \right)^{1/p}\left(B_F B_G \right)^{1/2q}.
\end{align*} 
\end{proof}

\subsection{Changing the Ingredients}

Like discrete Bessel multipliers \cite{xxlmult1}, a continuous
Bessel multiplier clearly depends on the chosen symbol, analysis
and synthesis functions. A natural question arises: What happens
if these items are changed? Are the frame multipliers similar to
each other if the symbol or the frames are similar to each other
(in the right similarity sense)? Do the multipliers depend 
continuously on the input data?
\par Let $m,m'\in L^\infty$ and $F,F',G,G'$ be Bessel functions. The
representation (\ref{rep1}) and linearity of the operators $T_F,
T_F', T_G, T_G', D_m$ and $D_m' $ result
\[\label{1}
\mathbf{M}_{m,F,G}-\mathbf{M}_{m',F,G}=T_G D_{m-m'}
T^*_F=\mathbf{M}_{m-m',F,G},
\]
\[\label{2}
\mathbf{M}_{m,F,G}-\mathbf{M}_{m,F',G}=T_G D_{m}
T^*_{F-F'}=\mathbf{M}_{m,F-F',G},
\]
\[\label{3}
\mathbf{M}_{m,F,G}-\mathbf{M}_{m,F,G'}=T_{G-G'} D_{m}
T^*_{F}=\mathbf{M}_{m,F,G-G'}.
\]

By adapting the methods in \cite{xxlmult1} and using the above
identities, we can prove the following theorem about continuous
Bessel multipliers.

\begin{thm} \label{sec:frammulprop1}
Let $F$ and $G$ be Bessel mappings for $\h$ with
respect to $(\Omega,\mu)$ and $m:\Omega\rightarrow \mathbb{C}$ be
a measurable function. Let $m^{(n)}$ be
functions indexed by $n\in\mathbb{N}$ with $m^{(n)} \to m$
in $L^{p}(\Omega, \mu)$. Then $\mathbf{M}_{m^{(n)},F,G}$ converges to $\mathbf{M}_{m,F,G}$ 
in the Schatten-$p$-norm, i.e. 
$ \| \mathbf{M}_{m^{(n)},F,G} - \mathbf{M}_{m,F,G} \|_{\mathcal{S}_p}  \rightarrow 0 $, as
$n\rightarrow\infty$.
\end{thm}
\begin{proof}
The proof follows immediately from (\ref{1}) and the norm estimate in Lemma \ref{tar} and Theorem \ref{sec:schatten1}.
\end{proof}
In particular this is also valid for trace class ($p = 1$) operators and bounded operators ($p = \infty$).

\begin{thm}
Let $m \in L^2(\Omega, \mu)$. Let $F$ and $G$ be Bessel mappings for $\h$.
Let $F^{(n)}$ be a sequence of Bessel mappings such that $F^{(n)}(\omega) \rightarrow F(\omega)$ in a uniform strong sense. 
Then $\mathbf{M}_{m,F^{(n)},G}$ converges to $\mathbf{M}_{m,F,G}$ in operator norm.
\end{thm}
\begin{proof}
Let $f,g \in \h$. For given $\epsilon > 0$, choose $N$ such that $\| F^{(n)}(\omega) - F(\omega) \| \leq \epsilon$ for all $n \geq N$, for all $\omega \in \Omega$. Then
\begin{align*}
| \langle (\mathbf{M}_{m,F^{(n)},G} & - \mathbf{M}_{m,F,G})f, g \rangle | \\
& \leq \int_{\Omega} | m(\omega) |  |\langle f, F^{(n)}(\omega) - F(\omega) \rangle   |  | \langle G(\omega), g \rangle  |\,d\mu(\omega) \\
& \leq \left( \int_{\Omega} | m(\omega) |^2  |\langle f, F^{(n)}(\omega) - F(\omega) \rangle   |^2 \,d\mu(\omega) \right)^{1/2}   \left( B_G \right)^{1/2} \| g \| \\
& \leq \epsilon \| m \|_{2} \|f \| \left( B_G \right)^{1/2} \| g \|. 
\end{align*}
Thus by Theorem \ref{murphy}
\[
\| \mathbf{M}_{m,F^{(n)},G}  - \mathbf{M}_{m,F,G} \| \leq \epsilon \| m \|_2 \left( B_G \right)^{1/2},
\]
so $\mathbf{M}_{m,F^{(n)},G}$ converges to $\mathbf{M}_{m,F,G}$ in operator norm.
\end{proof}

For symbols $m \in L^1(\Omega, \mu)$, we can find the following
theorem.
\begin{thm}
Let $m \in L^1(\Omega, \mu)$. Let $F$ and $G$ be Bessel mappings for $\h$ and $G$ be norm-bounded.
Let $F^{(n)}$ be a sequence of Bessel mappings such that $F^{(n)}(\omega) \rightarrow F(\omega)$ in a uniform strong sense. 
Then $\mathbf{M}_{m,F^{(n)},G}$ converges to $\mathbf{M}_{m,F,G}$ in operator norm.
\end{thm}
\begin{proof} For given $\epsilon > 0$, choose $N$ such that $\| F^{(n)}(\omega) - F(\omega) \| \leq \epsilon$ for all $n \geq N$, for all $\omega \in \Omega$. Then
\begin{align*}
\| (\mathbf{M}_{m,F^{(n)},G} - \mathbf{M}_{m,F,G})f \| & \le \int_{\Omega} | m(\omega) | \underbrace{| \langle f, F^{(n)}(\omega)  - F(\omega) \rangle |}_{\le \epsilon} \underbrace{\| G(\omega) \|}_{\le L_G} \,d\mu(\omega) \\ 
& \le \epsilon \| m \|_{1} L_G \| f\|.
\end{align*}
\end{proof}
In the last two results the roles of $F$ and $G$ can be switched.

\subsection{Examples: Continuous STFT and wavelet multipliers} \label{sec:STFTwavmult0}

Particular cases of continuous frame multipliers, that means multipliers for certain continuous frames, have already been studied and used before. In this section we briefly summarize some earlier results on STFT multipliers and Calder\'on-Toeplitz operators. 

\subsubsection{STFT multipliers}

Continuous frame multipliers have been discussed and extensively used earlier for the continuous frame of Definition \ref{D:Def_STFT}, i.e. the short-time Fourier transform. An operator  of the form
\[
\mathbf{M}_{m, \phi, \psi} f = \int_{-\infty}^{+\infty}\int_{-\infty}^{+\infty} m (a,b) \Psi_{\phi}(f)(a,b)
M_{b}T_{a}\psi\,dadb, 
\]
is called an STFT multiplier. In this context, the associated continuous frame multipliers are also known as time-frequency localization operators. They were first introduced and studied by Daubechies and Paul, \cite{da88}, \cite{dapa88}, where they are used as a mathematical tool to extract specific features of interest of a signal on phase space from its time-frequency representation.
The Wigner distribution constitutes a continuous frame that is essentially identical to the STFT, cf. \cite{wi32} or \cite{Fol}. It is closely related to the so-called Weyl calculus of quantum mechanics. In physics, multipliers for the Wigner distribution have been around for quite a long time in connection with questions of quantization, under the name "Anti-Wick operators" in the work of Berezin, \cite{be71}. They had also appeared earlier in the theory of pseudodifferential operators, cf. \cite{cofe78}. 
In these early works, the symbol is usually taken to be the characteristic function of some portion of the time-frequency plane. In \cite{rato93}, results on decay properties of the eigenvalues as well as smoothness of the eigenfunctions of Wigner multipliers with characteristic functions as symbols are derived.
A first result on Schatten class properties is contained in \cite{po66}, where it is shown for the Weyl correspondence that symbols in $L^2$ lead to Hilbert-Schmidt operators. Boundedness and mapping properties with respect to other Schatten classes of the correspondence between the symbol of a multiplier and the resulting operator are considered extensively in \cite{boco02}, \cite{bocogr04} (for Anti-Wick operators) and in \cite{cogr03}, \cite{cogr05} (for STFT multipliers). In these works, the operators are often interpreted as pseudodifferential operators.
In \cite{boco02}, it is shown that symbols in $L^p$ generate Anti-Wick operators in the Schatten $p$-class. This result is a special case of our Theorem \ref{sec:schatten1}. In \cite{bocogr04}, the theory of Anti-Wick operators is extended to symbols in distributional Sobolev spaces. 
The paper \cite{cogr03} can very well serve as a comprehensive first survey on localization operators, i.e. STFT multipliers. The theory is developed in the framework of time-frequency analysis, see also \cite{coro05}. As symbol classes so-called modulation spaces are considered. This requires that the window functions for the STFTs that form the continuous analysis and synthesis frames also belong to modulation spaces, usually to the Feichtinger algebra $\mathcal{S}_0$. In this case it is shown that symbols in the modulation space $M^{p, \infty}$ are sufficient for localization operators in the Schatten $p$-class, $1 \leq p < \infty$. Since $L^p$ spaces are continuously embedded in the modulation spaces $M^{p, \infty}$, this extends our results in the considered special case. In \cite{cogr05} the authors also present necessary conditions for Schatten classes. 
Symbolic calculus and Fredholm properties for localization operators are discussed in \cite{cogr06}. The PhD thesis \cite{ba10} is concerned with questions of approximation of operators by localization operators and density properties of the set of all localization operators with symbols in certain symbol classes in spaces of operators, equipped with different topologies.

\subsubsection{Calder\'on–-Toeplitz operators}

An operator defined by
\[
\mathbf{M}_{m,\psi} f = \int_{-\infty}^{+\infty}\int_{-\infty}^{+\infty} m (a,b) W_{\psi}(f)(a,b)
\psi^{a,b}\, \frac{dadb}{a^{2}}, 
\]
(in the notation of Definition \ref{D:Def_Wavelet}) is called a Calder\'on–-Toeplitz operator. This is a multiplier for the continuous frame given by the continuous wavelet transform. In this case, the function $m$ is referred to as the upper symbol of the operator, whereas the so-called lower symbol is given by $\tilde{m}(a,b) = \langle \mathbf{M}_{m, \psi} \psi^{a,b}, \psi^{a,b} \rangle$. The concept was first introduced in \cite{roch90} in 1990 as an analogue in terms of the wavelet transform to Toeplitz operators on spaces of analytic functions, for example Bergman spaces. The lower symbol corresponds analogously to the Berezin transform of a Toeplitz operator. Some interesting results on the spectral theory of these operators are shown in \cite{roch92}, for example the so-called correspondence principle, a statement on the dimensions of the spectral projections for certain bounded symbols. A number of mapping properties for Calder\'on-Toeplitz operators (with sufficiently smooth window function $\psi$) depending on the lower symbol are contained in \cite{now93}, for example the boundedness of the operator if and only the lower symbol is bounded, or the compactness of the operator if and only if the lower symbol vanishes at infinity. These are stronger versions of Theorems \ref{tar} and \ref{sec:compact2} in this specialized setting.
Some Schatten class properties for lower symbols in $L^p$, see Theorem \ref{sec:schatten1} for the general case, as well as as for positive upper symbols are proven. Eigenvalue estimates are given in \cite{roch92_2} and \cite{now93_2}. Calder\'on-Toeplitz operators are (along with STFT multipliers) proposed as a tool for time-frequency localization in \cite{da92}. A unified treatment of the elementary theory of STFT multipliers and wavelet transform multipliers (based on the underlying group structures) is given in the textbook \cite{wong}.   

\section{Controlled and weighted continuous frames }

The notion of controlled and weighted frames as introduced in \cite{petant} for discrete frames are closely linked to multipliers.
 So here we look at the corresponding properties for continuous frames.

\subsection{Controlled continuous frames}

\begin{defn}
Let $C\in GL(\h)$. A $C$-controlled continuous frame is a map
$F:\Omega\rightarrow\h$ such that there exist $m_{CL}>0$ and
$M_{CL}<\infty$ such that
\[
m_{CL}\|f\|^2\leq\int_{\Omega}\langle f,F(\omega)\langle
CF(\omega),f\rangle d\mu\leq M_{CL}\|f\|^2 \quad ( f\in\h ).
\]

\end{defn}
We call $L_C f=\int_{\Omega}\langle f,F(\omega)\rangle CF(\omega)
d\mu$ (in weak sense) the controlled continuous frame operator.
Analogue to Proposition 2.4 of \cite{petant} one can show that
$L_C\in GL(\h)$.
\begin{prop}
Let $F:\Omega\rightarrow\h$ be a $C$-controlled continuous frame for
some $C\in GL(\h)$. Then $F$ is a continuous frame for $\h$.
\end{prop}
\begin{proof}
Since $C$ is linear we have
\[
S_F f=\int_{\Omega}\langle f,F(\omega)\rangle F(\omega)
d\mu=C^{-1}\int_{\Omega}\langle f,F(\omega)\rangle CF(\omega)
d\mu=C^{-1} L_C f.
\]
Therefore $S_F$ is a bounded, positive and invertible operator and so
$F$ is a continuous frame.
\end{proof}
By definition $L_C$ is positive and $L_C=CS_F=S_F C^*$. Therefore it is easy to show that, given $C\in GL(\h)$ is a self-adjoint operator, then the mapping $F$ is a
$C$-controlled frame if and only if it is a continuous frame for
$\h$, $C$ is positive and commutes with $S_F$.

The following proposition shows that we can retrieve a continuous frames multiplier from a multiplier
of controlled frames. Actually, the role played by controlled
operators is that of a precondition matrices. 
\begin{prop}
Let $C,D\in GL(\h)$ be self-adjoint operators. If $F$ and $G$ are $C$-
respectively $D$-controlled frames and $\mathbb{M}$ is their multiplier
operator with respect to $m$, then $D^{-1}\mathbb{M}C^{-1}=
\mathbf{M}_{m,F,G}$.
\end{prop}
\begin{proof}
It is easy to see that for the $C$ and $D$ controlled frames $F$ and
 $G$, we have $T_C=CT$ and $T^*_D=T^* D$. Now the representation
(\ref{rep1}) results $D^{-1}\mathbb{M}C^{-1}= \mathbf{M}_{m,F,G}$.
\end{proof}

\subsection{Weighted continuous frames}

\begin{defn}
Let $\h $ be a complex Hilbert space and $(\Omega ,\mu)$ be a
measure space with positive measure $\mu$ and $m:\Omega\rightarrow
\mathbb{R}^+$. The mapping $F:\Omega\to\h$ is called a weighted
continuous
frame with respect to $(\Omega ,\mu)$ and $m$, if\\
\begin{enumerate}
\item  $F$ is weakly-measurable and $m$ is measurable; \\
\item  there exist constants $A, B> 0$ such that
\begin{equation}\label{wecof}
A\|f\|^{2}\leq \int_{\Omega}m(\omega)|\langle
f,F({\omega})\rangle|^{2}\,d\mu(\omega) \leq B\|f\|^{2}, \quad (
f\in \h).
\end{equation}
\end{enumerate}
 The mapping $F$ is called  weighted \emph{Bessel} if
the second inequality in (\ref{wecof}) holds.
\end{defn}

By using some ideas of \cite{StBa}, we have the following result.
\begin{thm}
Let $\mathbf{M}_{m,F,G}$ be invertible. Then:
\begin{enumerate}
\item  If $G$ is a Bessel map, then $\overline{m}F$ satisfies the lower frame condition.\\
\item  If $F$ is a Bessel map, then $mG$ satisfies the lower frame condition.\\
\item If $F$ and $mG$ (respectively $G$ and $\overline{m}F$) are Bessel maps, then they are continuous frames.\\
\item If $G$ is a Bessel map and $m\in L^\infty$, $m \neq 0$, then $F$ has a lower
bound.\\
\item If $F$ and $G$ are Bessel maps and $m\in L^\infty$, $m\neq 0$,
then both of $F$ and $G$ are continuous frames.
\end{enumerate}
\end{thm}
\begin{proof}
\begin{enumerate}
\item  For $f,g\in\h$, we have
\[
|\langle\mathbf{M}_{m,F,G}f,g\rangle|\leq\left(\int_{\Omega}|\langle
f,(\overline{m}F)(\omega)\rangle|^{2}
\,d\mu(\omega)\right)^{\frac{1}{2}}\left(\int_{\Omega}|\langle
G(\omega), g \rangle|^{2}\, d\mu(\omega)\right)^{\frac{1}{2}}
\]
without loss of generality, we can assume $f\neq 0$ and
\[
\int_{\Omega}|\langle f,(\overline{m}F)(\omega)\rangle|^{2}
\,d\mu(\omega)<\infty.
\] So
\[
|\langle\mathbf{M}_{m,F,G}f,g\rangle|\leq\sqrt{B_G}\,\|g\|\left(\int_{\Omega}|\langle
f,(\overline{m}F)(\omega)\rangle|^{2}
\,d\mu(\omega)\right)^{\frac{1}{2}}.
\]
By letting $g=(\mathbf{M}^*_{m,F,G})^{-1}f$ we have
\begin{eqnarray*}
\|f\|^2 & \leq&\sqrt{B_G}\,\|(\mathbf{M}^*_{m,F,G})^{-1}
\|\|f\|\left(\int_{\Omega}|\langle
f,(\overline{m}F(\omega))\rangle|^{2}
\,d\mu(\omega)\right)^{\frac{1}{2}}.
\end{eqnarray*}
So
\[
\frac{1}{\sqrt{B_G}\,\|(\mathbf{M}^*_{m,F,G})^{-1}
\|}\|f\|\leq\left(\int_{\Omega}|\langle
f,(\overline{m}F)(\omega)\rangle|^{2}
\,d\mu(\omega)\right)^{\frac{1}{2}}.
\]\\
\item  Similar to (1).
\item
Let $F$ be a Bessel map, then by part (1), $mG$ has a lower bound and so
it is a frame. If $mG$ is a Bessel map then by (2) $1 \cdot F = F$ satisfies the lower frame inquality and therefore is a frame.
\item
By (1) $\overline{m} F$ satisfies the lower frame inequality. Therefore
$$A \| f \|^2 \le \int_{\Omega}|\langle f,(\overline{m}F)(\omega)\rangle|^{2}
\,d\mu(\omega) \le \| m \|_\infty^2 \int_{\Omega}|\langle f, F(\omega)\rangle|^{2}
\,d\mu(\omega) . $$
And so
$$\frac{A}{\| m \|_\infty^2}\le \int_{\Omega}|\langle f, F(\omega)\rangle|^{2}
\,d\mu(\omega) . $$
\item Follows from (1), (2) and (3).
\end{enumerate}
\end{proof}

The following theorem finds a dual of a 
continuous frame in the case that the multiplier operator is
invertible. (Analogous to the discrete results in \cite{stobalrep11}).
\begin{thm}
Let $\mathbf{M}_{m,F,G}$ be invertible and $G$ be a continuous frame.
Then 
$(\mathbf{M}_{m,F,G}^{-1})^{*}\overline{m}F$ is a dual.
\end{thm}
\begin{proof} By replacing $f$ with  $\mathbf{M}_{m,F,G}^{-1}f$ in
\[
\langle \mathbf{M}_{m,F,G}f,g\rangle=\int_{\Omega}m(\omega)\langle
f, F(\omega)\rangle\langle G(\omega),g\rangle d\mu
\]
we get
\begin{eqnarray*}
\langle f,g\rangle&=&\int_{\Omega}m(\omega)\langle
\mathbf{M}_{m,F,G}^{-1}f, F(\omega)\rangle\langle G(\omega),g\rangle
d\mu \\&=&\int_{\Omega}\langle f,(\mathbf{M}_{m,F,G}^{-1})^*
\overline{m(\omega)} F(\omega)\rangle\langle G(\omega),g\rangle
d\mu.
\end{eqnarray*}
\end{proof}

\textbf{Acknowledgment}: Some of the results in this paper were
obtained when A. Rahimi visited the Acoustics Research
Institute, Austrian Academy of Sciences, Austria. He thanks this
institute for their hospitality.
\\This work was supported by the WWTF project MULAC ('Frame Multipliers: Theory and Application in Acoustics; MA07-025).


\end{document}